\documentclass[oneside,english]{amsart}
\usepackage[T1]{fontenc}
\usepackage[latin9]{inputenc}
\usepackage{color}
\usepackage{verbatim}
\usepackage{mathrsfs}
\usepackage{amstext}
\usepackage{amsthm}
\usepackage{amssymb}
\usepackage{esint}

\makeatletter
\numberwithin{equation}{section}
\numberwithin{figure}{section}
\theoremstyle{plain}
\newtheorem*{thm*}{\protect\theoremname}
\theoremstyle{plain}
\newtheorem*{cor*}{\protect\corollaryname}
\theoremstyle{plain}
\newtheorem{thm}{\protect\theoremname}
\theoremstyle{plain}

\theoremstyle{plain}
\newtheorem{lem}[thm]{\protect\lemmaname}
\theoremstyle{remark}
\newtheorem{rem}[thm]{\protect\remarkname}
\theoremstyle{definition}

\makeatother

\usepackage{babel}
\providecommand{\corollaryname}{Corollary}
\providecommand{\examplename}{Example}
\providecommand{\lemmaname}{Lemma}
\providecommand{\remarkname}{Remark}
\providecommand{\theoremname}{Theorem}

\begin{document}
\title[On the Carath\'eodory form in higher-order theory]{On the Carath\'eodory form in higher-order variational field theory}
\author{Zbyn\v{e}k Urban}
\address[Zbyn\v{e}k Urban]{Dpt. of Mathematics, Faculty of Civil Engineering, V\v{S}B-Technical University of Ostrava\\
Ludv\'{i}ka Pod\'{e}\v{s}t\v{e} 1875/17, 708 33 Ostrava,
Czech Republic}
\email[Corresponding author]{zbynek.urban@vsb.cz}
\author{Jana Voln\'a}
\address[Jana Voln\'a]{Dpt. of Mathematics, Faculty of Civil Engineering, V\v{S}B-Technical University of Ostrava\\
Ludv\'{i}ka Pod\'{e}\v{s}t\v{e} 1875/17, 708 33 Ostrava,
Czech Republic}
\email{jana.volna@vsb.cz}

\begin{abstract}
The Carath\'eodory form of the calculus of variations belongs to
the class of Lepage equivalents of first-order Lagrangians in field
theory. Here, this equivalent is generalized for second- and higher-order
Lagrangians by means of intrisic geometric operations applied to the
well-known Poincar\'e\textendash Cartan form and principal component
of Lepage forms, respectively. \textcolor{black}{For second-order
theory, our definition coincides with the previous result obtained
by Crampin and Saunders in a~different way.} The Carath\'eodory
equivalent of the Hilbert Lagrangian in general relativity is discussed.
\end{abstract}


\maketitle

\section{Introduction\label{sec:1}}

In this note, we describe a~generalization of the Carath\'eodory
form of the calculus of variations in \emph{second-order }and, for
specific Lagrangians, in\emph{ higher-order} field theory. Our approach
is based on a geometric relationship between the Poincar\'e\textendash Cartan
and Carath\'eodory forms,\emph{ }and analysis of the corresponding
global properties. In \cite{CramSaun2}, Crampin and Saunders obtained
the Carath\'eodory form for second-order Lagrangians as a~certain
projection onto a~sphere bundle. Here, we confirm this result by
means of a different, straightforward method which furthermore allows
higher-order generalization. It is a~standard fact in the global
variational field theory that the local expressions,
\begin{equation}
\Theta_{\lambda}=\mathscr{L}\omega_{0}+\sum_{k=0}^{r-1}\left(\sum_{l=0}^{r-1-k}(-1)^{l}d_{p_{1}}\ldots d_{p_{l}}\frac{\partial\mathscr{L}}{\partial y_{j_{1}\ldots j_{k}p_{1}\ldots p_{l}i}^{\sigma}}\right)\omega_{j_{1}\ldots j_{k}}^{\sigma}\wedge\omega_{i},\label{eq:PoiCar}
\end{equation}
which generalize the well-known Poincar\'e\textendash Cartan form
of the calculus of variations, define, in general, differential form
$\Theta_{\lambda}$ globally for Lagrangians $\lambda=\mathscr{L}\omega_{0}$
of order $r=1$ and $r=2$ only; see Krupka \cite{Krupka-Lepage}
($\Theta_{\lambda}$ is known as the principal component of a~Lepage
equivalent of Lagrangian $\lambda$), and Hor\'ak and Kol\'a\v{r}
\cite{HorakKolar} (for higher-order Poincar\'e\textendash Cartan
morphisms). We show that if $\Theta_{\lambda}$ is globally defined
differential form for a~\emph{class} of Lagrangians of order $r\geq3$,
then a~higher-order Carath\'eodory equivalent for Lagrangians belonging
to this class naturally arises by means of geometric operations acting
on $\Theta_{\lambda}$. To this purpose, for order $r=3$ we analyze
conditions, which describe the obstructions for globally defined principal
components of Lepage equivalents \eqref{eq:PoiCar} (or, higher\textendash order
Poincar\'e\textendash Cartan forms). 

The above-mentioned differential forms are examples of \emph{Lepage
forms}; for a~comprehensive exposition and original references see
Krupka \cite{Handbook,Krupka-Book}\emph{. }Similarly as the well-known
Cartan form describes analytical mechanics in a~coordinate-independent
way, in variational field theory (or, calculus of variations for multiple-integral
problems) this role is played by Lepage forms, in general. These objects
define the same variational functional as it is prescribed by a~given
Lagrangian and, moreover, variational properties (as variations, extremals,
or Noether's type invariance) of the corresponding functional are
globally characterized in terms of geometric operations (such as the
exterior derivative and the Lie derivative) acting on integrands -
the Lepage equivalents of a~Lagrangian.

A concrete application of our result in second-order field theory
includes the Carath\'eodory equivalent of the Hilbert Lagrangian
in general relativity, which we determine and it will be further studied
in future works.

Basic underlying structures, well adapted to this paper, can be found
in Voln\'a and Urban \cite{Volna}. If $(U,\varphi)$, $\varphi=(x^{i})$,
is a chart on smooth manifold $X$, we set
\begin{equation*}
{\color{black}{\color{black}\omega_{0}}=dx^{1}\wedge\ldots\wedge dx^{n},}\qquad{\color{black}\omega_{j}}{\color{black}=i_{\partial/\partial x^{j}}\omega_{0}=\frac{1}{(n-1)!}\varepsilon_{ji_{2}\ldots i_{n}}dx^{i_{2}}\wedge\ldots\wedge dx^{i_{n}},}
\end{equation*}
where $\varepsilon_{i_{1}i_{2}\ldots i_{n}}$ is the Levi-Civita permutation
symbol. If $\pi:Y\rightarrow X$ is a~fibered manifold and $W$ an
open subset of $Y$, then there exists a~unique morphism $h:\Omega^{r}W\rightarrow\Omega^{r+1}W$
of exterior algebras of differential forms such that for any fibered
chart $(V,\psi)$, $\psi=(x^{i},y^{\sigma})$, where $V\subset W$,
and any differentiable function $f:W^{r}\rightarrow\mathbb{R}$, where
$W^{r}=(\pi^{r,0})^{-1}(W)$ and $\pi^{r,s}:J^{r}Y\rightarrow J^{s}Y$
the jet bundle projection, 
\[
hf=f\circ\pi^{r+1,r},\quad\quad hdf=(d_{i}f)dx^{i},
\]
where
\begin{equation}
d_{i}=\frac{\partial}{\partial x^{i}}+\sum_{j_{1}\leq\ldots\leq j_{k}}\frac{\partial}{\partial y_{j_{1}\ldots j_{k}}^{\sigma}}y_{j_{1}\ldots j_{k}i}^{\sigma}\label{eq:FormalDerivative}
\end{equation}
is the $i$-th formal derivative operator associated with $(V,\psi)$.
A~differential form $q$-form $\rho\in\Omega_{q}^{r}W$ satisfying
$h\rho=0$ is called \emph{contact}, and $\rho$ is generated by contact
$1$-forms
\begin{equation*}
\omega_{j_{1}\ldots j_{k}}^{\sigma}=dy_{j_{1}\ldots j_{k}}^{\sigma}-y_{j_{1}\ldots j_{k}s}^{\sigma}dx^{s},\qquad0\leq k\leq r-1.
\end{equation*}
Throughout, we use the standard geometric concepts: the exterior derivative
$d$, the contraction $i_{\Xi}\rho$ and the Lie derivative $\partial_{\Xi}\rho$
of a differential form $\rho$ with respect to a vector field $\Xi$,
and the pull-back operation $*$ acting on differential forms.

\section{Lepage equivalents in first- and second-order field theory}

By a~\textit{Lagrangian} $\lambda$ for a~fibered manifold $\pi:Y\rightarrow X$
of order $r$ we mean an element of the submodule $\Omega_{n,X}^{r}W$
of $\pi^{r}$-horizontal $n$-forms in the module of $n$-forms $\Omega_{n}^{r}W$,
defined on an open subset $W^{r}$ of the $r$-th jet prolongation
$J^{r}Y$. In a~fibered chart $(V,\psi)$, $\psi=(x^{i},y^{\sigma})$,
where $V\subset W$, Lagrangian $\lambda\in\Omega_{n,X}^{r}W$ has
an expression
\begin{equation}
\lambda=\mathscr{L}\omega_{0},\label{eq:Lagrangian}
\end{equation}
where $\omega_{0}=dx^{1}\wedge dx^{2}\wedge\ldots\wedge dx^{n}$ is
the (local) volume element, and $\mathscr{L}:V^{r}\rightarrow\mathbb{R}$
is the \textit{Lagrange function} associated to $\lambda$ and $(V,\psi)$.

An $n$-form $\rho\in\Omega_{n}^{s}W$ is called a\textit{~Lepage
equivalent} of $\lambda\in\Omega_{n,X}^{r}W$, if the following two
conditions are satisfied: 

(i) $(\pi^{q,s+1})^{*}h\rho=(\pi^{q,r})^{*}\lambda$ (i.e. $\rho$
is \textit{equivalent }with $\lambda$), and 

(ii) $hi_{\xi}d\rho=0$ for arbitrary $\pi^{s,0}$-vertical vector
field $\xi$ on $W^{s}$ (i.e. $\rho$ is a\textit{~Lepage form}). 

The following theorem describe the structure of the Lepage equivalent
of a~Lagrangian (see \cite{Krupka-Lepage,Handbook}).
\begin{thm}
\label{Thm:LepageEquiv}Let $\lambda\in\Omega_{n,X}^{r}W$ be a~Lagrangian
of order $r$ for $Y$, locally expressed by \eqref{eq:Lagrangian}
with respect to a~fibered chart $(V,\psi)$, $\psi=(x^{i},y^{\sigma})$.
An $n$-form $\rho\in\Omega_{n}^{s}W$ is a~Lepage equivalent of
$\lambda$ if and only if
\begin{equation}
(\pi^{s+1,s})^{*}\rho=\Theta_{\lambda}+d\mu+\eta,\label{eq:LepDecomp}
\end{equation}
where $n$-form $\Theta_{\lambda}$ is defined on $V^{2r-1}$ by \emph{\eqref{eq:PoiCar},}
$\mu$ is a~contact $(n-1)$-form, and an $n$-form $\eta$ has the
order of contactness $\geq2$.
\end{thm}

$\Theta_{\lambda}$ is called the \emph{principal component} of the
Lepage form $\rho$ with respect to fibered chart $(V,\psi)$. In
general, decomposition \eqref{eq:LepDecomp} is \emph{not} uniquely
determined with respect to contact forms $\mu$, $\eta$, and the
principal component $\Theta_{\lambda}$ need \emph{not} define a~global
form on $W^{2r-1}$. Nevertheless, the Lepage equivalent $\rho$ satisfying
\eqref{eq:LepDecomp} is globally defined on $W^{s}$; moreover $E_{\lambda}=p_{1}d\rho$
is a~globally defined $(n+1)$-form on $W^{2r}$, called the \emph{Euler\textendash Lagrange
form} associated to $\lambda$.

We recall the known examples of Lepage equivalents of first- and second-order
Lagrangians, determined by means of additional requirements.
\begin{lem}
\textbf{\textup{\label{Lem:PrincipalLepEq}(Principal Lepage form)}}
\emph{(a)} For every Lagrangian $\lambda$ of order $r=1$, there
exists a unique Lepage equivalent $\Theta_{\lambda}$ of $\lambda$
on $W^{1}$, which is $\pi^{1,0}$-horizontal and has the order of
contactness $\leq1$. In a fibered chart $(V,\psi)$, $\Theta{}_{\lambda}$
has an expression 
\begin{equation}
\Theta_{\lambda}=\mathscr{L}\omega_{0}+\frac{\partial\mathscr{L}}{\partial y_{j}^{\sigma}}\omega^{\sigma}\wedge\omega_{j}.\label{eq:Poincare-Cartan}
\end{equation}

\emph{(b)} For every Lagrangian $\lambda$ of order $r=2$, there
exists a unique Lepage equivalent $\Theta_{\lambda}$ of $\lambda$
on $W^{3}$, which is $\pi^{3,1}$-horizontal and has the order of
contactness $\leq1$. In a fibered chart $(V,\psi)$, $\Theta{}_{\lambda}$
has an expression 

\begin{equation}
\Theta_{\lambda}=\mathscr{L}\omega_{0}+\left(\frac{\partial\mathscr{L}}{\partial y_{j}^{\sigma}}-d_{p}\frac{\partial\mathscr{L}}{\partial y_{pj}^{\sigma}}\right)\omega^{\sigma}\wedge\omega_{j}+\frac{\partial\mathscr{L}}{\partial y_{ij}^{\sigma}}\omega_{i}^{\sigma}\wedge\omega_{j}.\label{eq:Poincare-Cartan-2ndOrder}
\end{equation}
\end{lem}

For $r=1$ and $r=2$, the principal component $\Theta_{\lambda}$
\eqref{eq:PoiCar} is a~\emph{globally defined} Lepage equivalent
of $\lambda$. We point out that for $r\geq3$ this is \emph{not}
true (see \cite{HorakKolar,Krupka-Lepage}). \eqref{eq:Poincare-Cartan}
is the well-known \emph{Poincar\'e-Cartan form} (cf. Garc\'ia \cite{Garcia}),
and it is generealized for second-order Lagrangians by globally defined
\emph{principal Lepage equivalent} \eqref{eq:Poincare-Cartan-2ndOrder}
on $W^{3}\subset J^{3}Y$.
\begin{lem}
\textbf{\textup{\label{Lem:Fundamental}(Fundamental Lepage form)}}
Let $\lambda\in\Omega_{n,X}^{1}W$ be a~Lagrangian of order 1 for
$Y$, locally expressed by \eqref{eq:Lagrangian}. There exists a~unique
Lepage equivalent $Z_{\lambda}\in\Omega_{n}^{1}W$ of $\lambda$,
which satisfies $Z_{h\rho}=(\pi^{1,0})^{*}\rho$ for any $n$-form
$\rho\in\Omega_{n}^{0}W$ on $W$ such that $h\rho=\lambda$. With
respect to a fibered chart $(V,\psi)$, $Z_{\lambda}$ has an expression
\begin{align}
Z_{\lambda} & =\mathscr{L}\omega_{0}+\sum_{k=1}^{n}\frac{1}{(n-k)!}\frac{1}{(k!)^{2}}\frac{\partial^{k}\mathscr{L}}{\partial y_{j_{1}}^{\sigma_{1}}\ldots\partial y_{j_{k}}^{\sigma_{k}}}\varepsilon_{j_{1}\ldots j_{k}i_{k+1}\ldots i_{n}}\label{eq:Fundamental}\\
 & \quad\cdot\omega^{\sigma_{1}}\land\ldots\wedge\omega^{\sigma_{k}}\wedge dx^{i_{k+1}}\wedge\ldots\wedge dx^{i_{n}}.\nonumber 
\end{align}
\end{lem}

$Z_{\lambda}$ \eqref{eq:Fundamental} is known as the \emph{fundamental
Lepage form} \cite{Krupka-Fund.Lep.eq.}, \cite{Betounes}), and it
is characterized by the equivalence: $Z_{\lambda}$ is closed if and
only if $\lambda$ is trivial (i.e. the Euler\textendash Lagrange
expressions associated with $\lambda$ vanish identically). Recently,
the form \eqref{eq:Fundamental} was studied for variational problems
for submanifolds in \cite{UrbBra}, and applied for studying symmetries
and conservation laws in \cite{Javier}.
\begin{lem}
\textbf{\textup{\label{Lem:Caratheodory}(Carath\'eodory form) }}Let
$\lambda\in\Omega_{n,X}^{1}W$ be a~non-vanishing Lagrangian of order
1 for $Y$ \eqref{eq:Lagrangian}. Then a~differential $n$-form
$\Lambda_{\lambda}\in\Omega_{n}^{1}W$, locally expressed as
\begin{align}
\Lambda_{\lambda} & =\frac{1}{\mathscr{L}^{n-1}}\bigwedge_{j=1}^{n}\left(\mathscr{L}dx^{j}+\frac{\partial\mathscr{L}}{\partial y_{j}^{\sigma}}\omega^{\sigma}\right)\label{eq:CaratheodoryForm}\\
 & =\frac{1}{\mathscr{L}^{n-1}}\left(\mathscr{L}dx^{1}+\frac{\partial\mathscr{L}}{\partial y_{1}^{\sigma_{1}}}\omega^{\sigma_{1}}\right)\wedge\ldots\wedge\left(\mathscr{L}dx^{n}+\frac{\partial\mathscr{L}}{\partial y_{n}^{\sigma_{n}}}\omega^{\sigma_{n}}\right),\nonumber 
\end{align}
is a Lepage equivalent of $\lambda$.
\end{lem}

$\Lambda_{\lambda}$ \eqref{eq:Fundamental} is the well-known \emph{Carath\'eodory
form} (cf. \cite{Caratheodory}), associated to Lagrangian $\lambda\in\Omega_{n,X}^{1}W$,
which is nowhere zero. $\Lambda_{\lambda}$ is uniquely characterized
by the following properties: $\Lambda_{\lambda}$ is (i) a~Lepage
equivalent of $\lambda$, (ii) decomposable, (iii) $\pi^{1,0}$-horizontal
(i.e. semi-basic with respect to projection $\pi^{1,0}$).

\section{The Carath\'eodory form: second-order generalization}

Let $\lambda\in\Omega_{n,X}^{1}W$ be a~\emph{non-vanishing,} \emph{first-order}
Lagrangian on $W^{1}\subset J^{1}Y$. In the next lemma, we describe
a~new observation, showing that the Carath\'eodory form $\Lambda_{\lambda}$
\eqref{eq:CaratheodoryForm} arises from the Poincar\'e-Cartan form
$\Theta_{\lambda}$ \eqref{eq:Poincare-Cartan} by means of contraction
operations on differential forms with respect to the formal derivative
vector fields $d_{i}$ \eqref{eq:FormalDerivative}.
\begin{lem}
\label{lem:Car-PC}The Carath\'eodory form $\Lambda_{\lambda}$ \eqref{eq:CaratheodoryForm}
and the Poincar\'e-Cartan form $\Theta_{\lambda}$ \eqref{eq:Poincare-Cartan}
satisfy
\begin{align*}
\Lambda_{\lambda} & =\frac{1}{\mathscr{L}^{n-1}}\bigwedge_{j=1}^{n}(-1)^{n-j}i_{d_{n}}\ldots i_{d_{j+1}}i_{d_{j-1}}\ldots i_{d_{1}}\Theta_{\lambda}.
\end{align*}
\end{lem}

\begin{proof}
From the decomposable structure of $\Lambda_{\lambda}$, we see that
what is needed to show is the formula
\begin{equation*}
i_{d_{n}}\ldots i_{d_{j+1}}i_{d_{j-1}}\ldots i_{d_{1}}\Theta_{\lambda}=(-1)^{n-j}\left(\mathscr{L}dx^{j}+\frac{\partial\mathscr{L}}{\partial y_{j}^{\sigma}}\omega^{\sigma}\right)
\end{equation*}
for every $j$, $1\leq j\leq n$. Since $dx^{k}\wedge\omega_{j}=\delta_{j}^{k}\omega_{0}$,
the Poincar\'e-Cartan form is expressible as
\begin{align*}
\Theta_{\lambda} & =\mathscr{L}\omega_{0}+\frac{\partial\mathscr{L}}{\partial y_{j}^{\sigma}}\omega^{\sigma}\wedge\omega_{j}=\mathscr{L}\omega_{0}+\frac{\partial\mathscr{L}}{\partial y_{j}^{\sigma}}dy^{\sigma}\wedge\omega_{j}-\frac{\partial\mathscr{L}}{\partial y_{j}^{\sigma}}y_{k}^{\sigma}dx^{k}\wedge\omega_{j}\\
 & =\left(\mathscr{L}-\frac{\partial\mathscr{L}}{\partial y_{1}^{\sigma}}y_{1}^{\sigma}-\frac{\partial\mathscr{L}}{\partial y_{2}^{\sigma}}y_{2}^{\sigma}-\ldots-\frac{\partial\mathscr{L}}{\partial y_{n}^{\sigma}}y_{n}^{\sigma}\right)\omega_{0}+\frac{\partial\mathscr{L}}{\partial y_{j}^{\sigma}}dy^{\sigma}\wedge\omega_{j}.
\end{align*}
Applying the contraction operations to $\Theta_{\lambda}$, we obtain
by means of a straightforward computation for every $j$,
\begin{align*}
 & i_{d_{j-1}}\ldots i_{d_{1}}\Theta_{\lambda}=\left(\mathscr{L}-\frac{\partial\mathscr{L}}{\partial y_{j}^{\sigma}}y_{j}^{\sigma}-\ldots-\frac{\partial\mathscr{L}}{\partial y_{n}^{\sigma}}y_{n}^{\sigma}\right)dx^{j}\wedge\ldots\wedge dx^{n}\\
 & \quad\quad+(-1)^{j-1}\sum_{k=j}^{n}\frac{\partial\mathscr{L}}{\partial y_{k}^{\sigma}}dy^{\sigma}\wedge i_{d_{j-1}}\ldots i_{d_{1}}\omega_{k}\\
 & \quad\quad+\sum_{l=1}^{j-1}(-1)^{l-1}y_{l}^{\sigma}i_{d_{j-1}}\ldots i_{d_{l+1}}i_{d_{l-1}}\ldots i_{d_{1}}\left(\sum_{k=j}^{n}\frac{\partial\mathscr{L}}{\partial y_{k}^{\sigma}}\omega_{k}\right),
\end{align*}
and
\begin{align*}
 & i_{d_{j+1}}i_{d_{j-1}}\ldots i_{d_{1}}\Theta_{\lambda}\\
 & \quad=-\left(\mathscr{L}-\frac{\partial\mathscr{L}}{\partial y_{j}^{\sigma}}y_{j}^{\sigma}-\sum_{k=j+2}^{n}\frac{\partial\mathscr{L}}{\partial y_{k}^{\sigma}}y_{k}^{\sigma}\right)dx^{j}\wedge dx^{j+2}\wedge\ldots\wedge dx^{n}\\
 & \quad+(-1)^{j}\sum_{k=j}^{n}\frac{\partial\mathscr{L}}{\partial y_{k}^{\sigma}}dy^{\sigma}\wedge i_{d_{j+1}}i_{d_{j-1}}\ldots i_{d_{1}}\omega_{k}\\
 & \quad+\sum_{l=1}^{j-1}(-1)^{l-1}y_{l}^{\sigma}i_{d_{j+1}}i_{d_{j-1}}\ldots i_{d_{l+1}}i_{d_{l-1}}\ldots i_{d_{1}}\left(\frac{\partial\mathscr{L}}{\partial y_{j}^{\sigma}}\omega_{j}+\sum_{k=j+2}^{n}\frac{\partial\mathscr{L}}{\partial y_{k}^{\sigma}}\omega_{k}\right)\\
 & \quad+(-1)^{j-1}y_{j+1}^{\sigma}i_{d_{j-1}}\ldots i_{d_{1}}\left(\frac{\partial\mathscr{L}}{\partial y_{j}^{\sigma}}\omega_{j}+\sum_{k=j+2}^{n}\frac{\partial\mathscr{L}}{\partial y_{k}^{\sigma}}\omega_{k}\right).
\end{align*}
Following the inductive structure of the preceding expressions, we
get after the next $n-j-1$ steps,
\begin{align*}
 & i_{d_{n}}\ldots i_{d_{j+1}}i_{d_{j-1}}\ldots i_{d_{1}}\Theta_{\lambda}\\
 & \quad=(-1)^{n-j}\left(\mathscr{L}-\frac{\partial\mathscr{L}}{\partial y_{j}^{\sigma}}y_{j}^{\sigma}\right)dx^{j}+(-1)^{n-j}\frac{\partial\mathscr{L}}{\partial y_{j}^{\sigma}}dy^{\sigma}-(-1)^{n-j}\frac{\partial\mathscr{L}}{\partial y_{j}^{\sigma}}\sum_{k\neq j}y_{k}^{\sigma}dx^{k}\\
 & \quad=(-1)^{n-j}\left(\mathscr{L}dx^{j}+\frac{\partial\mathscr{L}}{\partial y_{j}^{\sigma}}\omega^{\sigma}\right),
\end{align*}
as required.
\end{proof}
An intrinsic nature of Lemma \ref{lem:Car-PC} indicates a~possible
extension of the Carath\'eodory form \eqref{eq:CaratheodoryForm}
for higher-order variational problems. We put
\begin{equation}
\Lambda_{\lambda}=\frac{1}{\mathscr{L}^{n-1}}\bigwedge_{j=1}^{n}(-1)^{n-j}i_{d_{n}}\ldots i_{d_{j+1}}i_{d_{j-1}}\ldots i_{d_{1}}\Theta_{\lambda},\label{eq:2nd-Caratheodory}
\end{equation}
where $\Theta_{\lambda}$ in \eqref{eq:2nd-Caratheodory} denotes
the principal Lepage equivalent \eqref{eq:Poincare-Cartan-2ndOrder}
of a~\emph{second-order} Lagrangian $\lambda$, and verify that formula
\eqref{eq:2nd-Caratheodory} defines a~global form.
\begin{thm}
\label{thm:Main}Let $\lambda\in\Omega_{n,X}^{2}W$ be a~non-vanishing
second-order Lagrangian on $W^{2}\subset J^{2}Y$. Then $\Lambda_{\lambda}$
satisfies:

\emph{(a)} Formula \eqref{eq:2nd-Caratheodory} defines an $n$-form
on $W^{3}\subset J^{3}Y$. 

\emph{(b)} If $\lambda\in\Omega_{n,X}^{2}W$ has an expression \eqref{eq:Lagrangian}
with respect to a~fibered chart $(V,\psi)$, $\psi=(x^{i},y^{\sigma})$,
such that $V\subset W$, then $\Lambda_{\lambda}$ \eqref{eq:2nd-Caratheodory}
is expressed by
\begin{align}
\Lambda_{\lambda} & =\frac{1}{\mathscr{L}^{n-1}}\bigwedge_{j=1}^{n}\left(\mathscr{L}dx^{j}+\left(\frac{\partial\mathscr{L}}{\partial y_{j}^{\sigma}}-d_{i}\frac{\partial\mathscr{L}}{\partial y_{ij}^{\sigma}}\right)\omega^{\sigma}+\frac{\partial\mathscr{L}}{\partial y_{ij}^{\sigma}}\omega_{i}^{\sigma}\right).\label{eq:2ndCaratheodoryExpression}
\end{align}

\emph{(c)} $\Lambda_{\lambda}\in\Omega_{n}^{3}W$ \eqref{eq:2nd-Caratheodory}
associated to a~second-order Lagrangian $\lambda\in\Omega_{n,X}^{2}W$
is a~Lepage equivalent of $\lambda$, which is decomposable and $\pi^{3,1}$-horizontal.
\end{thm}

\begin{proof}
1. Suppose $(V,\psi)$, $\psi=(x^{i},y^{\sigma})$, and $(\bar{V},\bar{\psi})$,
$\bar{\psi}=(\bar{x}^{i},\bar{y}^{\sigma})$, are two overlapping
fibered charts on $W$. For $\lambda\in\Omega_{n,X}^{2}W$, the corresponding
chart expressions $\lambda=\mathscr{L}\omega_{0}$ and $\lambda=\bar{\mathscr{L}}\bar{\omega}_{0}$
satisfy
\begin{equation}
\mathscr{L}=\left(\bar{\mathscr{L}}\circ\bar{\psi}^{-1}\circ\psi\right)\det\frac{\partial\bar{x}^{i}}{\partial x^{j}}.\label{eq:LagrangianTransform}
\end{equation}
Since the push-forward vector field $\bar{d}_{k}$,
\[
\bar{d}_{k}=\frac{\partial}{\partial\bar{x}^{k}}+\bar{y}_{k}^{\sigma}\frac{\partial}{\partial\bar{y}^{\sigma}}+\bar{y}_{kl}^{\sigma}\frac{\partial}{\partial\bar{y}_{l}^{\sigma}},
\]
of vector field $(\partial x^{i}/\partial\bar{x}^{k})d_{i}$ with
respect to the chart transformation $\bar{\psi}^{-1}\circ\psi$ satisfies
\[
(\bar{\psi}^{-1}\circ\psi)^{*}\left(i_{\bar{d}_{k}}\Theta_{\lambda}\right)=i_{\frac{\partial x^{i}}{\partial\bar{x}^{k}}d_{i}}\left((\bar{\psi}^{-1}\circ\psi)^{*}\Theta_{\lambda}\right)=\frac{\partial x^{i}}{\partial\bar{x}^{k}}i_{d_{i}}\left((\bar{\psi}^{-1}\circ\psi)^{*}\Theta_{\lambda}\right),
\]
and $\Theta_{\lambda}$ \eqref{eq:Poincare-Cartan-2ndOrder} is globally
defined, we get
\begin{align*}
 & (\bar{\psi}^{-1}\circ\psi)^{*}\frac{1}{\mathscr{\bar{L}}^{n-1}}\bigwedge_{j=1}^{n}(-1)^{n-j}i_{\bar{d}_{n}}\ldots i_{\bar{d}_{j+1}}i_{\bar{d}_{j-1}}\ldots i_{\bar{d}_{1}}\Theta_{\lambda}\\
 & \quad=\frac{1}{\left(\bar{\mathscr{L}}\circ\bar{\psi}^{-1}\circ\psi\right)^{n-1}}\bigwedge_{j=1}^{n}(-1)^{n-j}\frac{\partial x^{i_{1}}}{\partial\bar{x}^{1}}\ldots\frac{\partial x^{i_{j-1}}}{\partial\bar{x}^{j-1}}\frac{\partial x^{i_{j+1}}}{\partial\bar{x}^{j+1}}\ldots\frac{\partial x^{i_{n}}}{\partial\bar{x}^{n}}\\
 & \quad\quad\cdot i_{d_{i_{n}}}\ldots i_{d_{i_{j+1}}}i_{d_{i_{j-1}}}\ldots i_{d_{i_{1}}}\left((\bar{\psi}^{-1}\circ\psi)^{*}\Theta_{\lambda}\right)\\
 & \quad=\frac{1}{\left(\bar{\mathscr{L}}\circ\bar{\psi}^{-1}\circ\psi\right)^{n-1}}\left(\frac{\partial x^{i_{1}}}{\partial\bar{x}^{1}}\ldots\frac{\partial x^{i_{n}}}{\partial\bar{x}^{n}}\varepsilon_{i_{1}\ldots i_{n}}\right)^{n}\\
 & \quad\quad\cdot\bigwedge_{j=1}^{n}(-1)^{n-j}i_{d_{n}}\ldots i_{d_{j+1}}i_{d_{j-1}}\ldots i_{d_{1}}(\bar{\psi}^{-1}\circ\psi)^{*}\Theta_{\lambda}\\
 & \quad=\frac{1}{\left((\bar{\mathscr{L}}\circ\bar{\psi}^{-1}\circ\psi)\det\frac{\partial\bar{x}^{i}}{\partial x^{j}}\right)^{n-1}}\bigwedge_{j=1}^{n}(-1)^{n-j}i_{d_{n}}\ldots i_{d_{j+1}}i_{d_{j-1}}\ldots i_{d_{1}}(\bar{\psi}^{-1}\circ\psi)^{*}\Theta_{\lambda}\\
 & \quad=\frac{1}{\mathscr{L}^{n-1}}\bigwedge_{j=1}^{n}(-1)^{n-j}i_{d_{n}}\ldots i_{d_{j+1}}i_{d_{j-1}}\ldots i_{d_{1}}(\bar{\psi}^{-1}\circ\psi)^{*}\Theta_{\lambda}\\
 & \quad=\frac{1}{\mathscr{L}^{n-1}}\bigwedge_{j=1}^{n}(-1)^{n-j}i_{d_{n}}\ldots i_{d_{j+1}}i_{d_{j-1}}\ldots i_{d_{1}}\Theta_{\lambda},
\end{align*}
as required.

2. Analogously to the proof of Lemma \ref{lem:Car-PC}, we find a~chart
expression of $1$-form
\[
i_{d_{n}}\ldots i_{d_{j+1}}i_{d_{j-1}}\ldots i_{d_{1}}\Theta_{\lambda},
\]
where $\Theta_{\lambda}$ is the principal Lepage equivalent \eqref{eq:Poincare-Cartan-2ndOrder}.
Using $dx^{k}\wedge\omega_{j}=\delta_{j}^{k}\omega_{0}$, we have
\begin{align*}
\Theta_{\lambda} & =\left(\mathscr{L}-\left(\frac{\partial\mathscr{L}}{\partial y_{j}^{\sigma}}-d_{p}\frac{\partial\mathscr{L}}{\partial y_{pj}^{\sigma}}\right)y_{j}^{\sigma}-\frac{\partial\mathscr{L}}{\partial y_{ij}^{\sigma}}y_{ij}^{\sigma}\right)\omega_{0}\\
 & +\left(\frac{\partial\mathscr{L}}{\partial y_{j}^{\sigma}}-d_{p}\frac{\partial\mathscr{L}}{\partial y_{pj}^{\sigma}}\right)dy^{\sigma}\wedge\omega_{j}+\frac{\partial\mathscr{L}}{\partial y_{ij}^{\sigma}}dy_{i}^{\sigma}\wedge\omega_{j}.
\end{align*}
Then
\begin{align*}
 & i_{d_{j-1}}\ldots i_{d_{1}}\Theta_{\lambda}\\
 & \quad=\left(\mathscr{L}-\left(\frac{\partial\mathscr{L}}{\partial y_{k}^{\sigma}}-d_{p}\frac{\partial\mathscr{L}}{\partial y_{pk}^{\sigma}}\right)y_{k}^{\sigma}-\frac{\partial\mathscr{L}}{\partial y_{ik}^{\sigma}}y_{ik}^{\sigma}\right)dx^{j}\wedge\ldots\wedge dx^{n}\\
 & \quad+\sum_{l=1}^{j-1}(-1)^{l-1}\left(\left(\frac{\partial\mathscr{L}}{\partial y_{k}^{\sigma}}-d_{p}\frac{\partial\mathscr{L}}{\partial y_{pk}^{\sigma}}\right)y_{l}^{\sigma}+\frac{\partial\mathscr{L}}{\partial y_{ik}^{\sigma}}y_{il}^{\sigma}\right)i_{d_{j-1}}\ldots i_{d_{l+1}}i_{d_{l-1}}\ldots i_{d_{1}}\omega_{k}\\
 & \quad+(-1)^{j-1}\left(\left(\frac{\partial\mathscr{L}}{\partial y_{k}^{\sigma}}-d_{p}\frac{\partial\mathscr{L}}{\partial y_{pk}^{\sigma}}\right)dy^{\sigma}+\frac{\partial\mathscr{L}}{\partial y_{ik}^{\sigma}}dy_{i}^{\sigma}\right)\wedge\left(i_{d_{j-1}}\ldots i_{d_{1}}\omega_{k}\right)\\
 & \quad=\left(\mathscr{L}-\sum_{k=j}^{n}\left(\left(\frac{\partial\mathscr{L}}{\partial y_{k}^{\sigma}}-d_{p}\frac{\partial\mathscr{L}}{\partial y_{pk}^{\sigma}}\right)y_{k}^{\sigma}+\frac{\partial\mathscr{L}}{\partial y_{ik}^{\sigma}}y_{ik}^{\sigma}\right)\right)dx^{j}\wedge\ldots\wedge dx^{n}\\
 & \quad+\sum_{l=1}^{j-1}(-1)^{l-1}\sum_{k=j}^{n}\left(\left(\frac{\partial\mathscr{L}}{\partial y_{k}^{\sigma}}-d_{p}\frac{\partial\mathscr{L}}{\partial y_{pk}^{\sigma}}\right)y_{l}^{\sigma}+\frac{\partial\mathscr{L}}{\partial y_{ik}^{\sigma}}y_{il}^{\sigma}\right)i_{d_{j-1}}\ldots i_{d_{l+1}}i_{d_{l-1}}\ldots i_{d_{1}}\omega_{k}\\
 & \quad+(-1)^{j-1}\sum_{k=j}^{n}\left(\left(\frac{\partial\mathscr{L}}{\partial y_{k}^{\sigma}}-d_{p}\frac{\partial\mathscr{L}}{\partial y_{pk}^{\sigma}}\right)dy^{\sigma}+\frac{\partial\mathscr{L}}{\partial y_{ik}^{\sigma}}dy_{i}^{\sigma}\right)\wedge\left(i_{d_{j-1}}\ldots i_{d_{1}}\omega_{k}\right),
\end{align*}
and
\begin{align*}
 & i_{d_{j+1}}i_{d_{j-1}}\ldots i_{d_{1}}\Theta_{\lambda}\\
 & \quad=\Biggl(-\mathscr{L}+\sum_{\begin{array}{c}
k=j\\
k\neq j+1
\end{array}}^{n}\left(\left(\frac{\partial\mathscr{L}}{\partial y_{k}^{\sigma}}-d_{p}\frac{\partial\mathscr{L}}{\partial y_{pk}^{\sigma}}\right)y_{k}^{\sigma}+\frac{\partial\mathscr{L}}{\partial y_{ik}^{\sigma}}y_{ik}^{\sigma}\right)\Biggr)dx^{j}\wedge\bigwedge_{l=j+2}^{n}dx^{l}\\
 & \quad+\sum_{l=1}^{j-1}(-1)^{l-1}\sum_{\begin{array}{c}
k=j\\
k\neq j+1
\end{array}}^{n}\left(\left(\frac{\partial\mathscr{L}}{\partial y_{k}^{\sigma}}-d_{p}\frac{\partial\mathscr{L}}{\partial y_{pk}^{\sigma}}\right)y_{l}^{\sigma}+\frac{\partial\mathscr{L}}{\partial y_{ik}^{\sigma}}y_{il}^{\sigma}\right)\\
 & \quad\quad i_{d_{j+1}}i_{d_{j-1}}\ldots i_{d_{l+1}}i_{d_{l-1}}\ldots i_{d_{1}}\omega_{k}\\
 & \quad+(-1)^{j-1}\sum_{\begin{array}{c}
k=j\\
k\neq j+1
\end{array}}^{n}\left(\left(\frac{\partial\mathscr{L}}{\partial y_{k}^{\sigma}}-d_{p}\frac{\partial\mathscr{L}}{\partial y_{pk}^{\sigma}}\right)y_{j+1}^{\sigma}+\frac{\partial\mathscr{L}}{\partial y_{ik}^{\sigma}}y_{i,j+1}^{\sigma}\right)i_{d_{j-1}}\ldots i_{d_{1}}\omega_{k}\\
 & \quad+(-1)^{j}\sum_{k=j}^{n}\left(\left(\frac{\partial\mathscr{L}}{\partial y_{k}^{\sigma}}-d_{p}\frac{\partial\mathscr{L}}{\partial y_{pk}^{\sigma}}\right)dy^{\sigma}+\frac{\partial\mathscr{L}}{\partial y_{ik}^{\sigma}}dy_{i}^{\sigma}\right)\wedge\left(i_{d_{j+1}}i_{d_{j-1}}\ldots i_{d_{1}}\omega_{k}\right).
\end{align*}
After another $n-j-1$ steps we obtain
\begin{align*}
 & i_{d_{n}}\ldots i_{d_{j+1}}i_{d_{j-1}}\ldots i_{d_{1}}\Theta_{\lambda}\\
 & =(-1)^{n-j}\left(\mathscr{L}dx^{j}+\left(\frac{\partial\mathscr{L}}{\partial y_{j}^{\sigma}}-d_{p}\frac{\partial\mathscr{L}}{\partial y_{pj}^{\sigma}}\right)\omega^{\sigma}+\frac{\partial\mathscr{L}}{\partial y_{ij}^{\sigma}}\omega_{i}^{\sigma}\right).
\end{align*}

3. From \eqref{eq:2ndCaratheodoryExpression} it is evident that $\Lambda_{\lambda}$
\eqref{eq:2nd-Caratheodory} is decomposable, $\pi^{3,1}$-horizontal,
and obeys $h\Lambda_{\lambda}=\lambda$. It is sufficient to verify
that $\Lambda_{\lambda}$ is a~Lepage form, that is $hi_{\xi}d\Lambda_{\lambda}=0$
for arbitrary $\pi^{3,0}$-vertical vector field $\xi$ on $W^{3}\subset J^{3}Y$.
This follows, however, by means of a~straightforward computation
using chart expression \eqref{eq:2ndCaratheodoryExpression}. Indeed,
we have
\begin{align*}
 & d\Lambda_{\lambda}=(1-n)\frac{1}{\mathscr{L}^{n}}d\mathscr{L}\wedge\bigwedge_{j=1}^{n}\left(\mathscr{L}dx^{j}+\left(\frac{\partial\mathscr{L}}{\partial y_{j}^{\sigma}}-d_{i}\frac{\partial\mathscr{L}}{\partial y_{ij}^{\sigma}}\right)\omega^{\sigma}+\frac{\partial\mathscr{L}}{\partial y_{ij}^{\sigma}}\omega_{i}^{\sigma}\right)\\
 & \quad+\frac{1}{\mathscr{L}^{n-1}}\sum_{k=1}^{n}(-1)^{k-1}d\left(\mathscr{L}dx^{k}+\left(\frac{\partial\mathscr{L}}{\partial y_{k}^{\sigma}}-d_{i}\frac{\partial\mathscr{L}}{\partial y_{ik}^{\sigma}}\right)\omega^{\sigma}+\frac{\partial\mathscr{L}}{\partial y_{ik}^{\sigma}}\omega_{i}^{\sigma}\right)\\
 & \quad\wedge\bigwedge_{j\neq k}\left(\mathscr{L}dx^{j}+\left(\frac{\partial\mathscr{L}}{\partial y_{j}^{\sigma}}-d_{i}\frac{\partial\mathscr{L}}{\partial y_{ij}^{\sigma}}\right)\omega^{\sigma}+\frac{\partial\mathscr{L}}{\partial y_{ij}^{\sigma}}\omega_{i}^{\sigma}\right),
\end{align*}
and the contraction of $d\Lambda_{\lambda}$ with respect to $\pi^{3,0}$-vertical
vector field $\xi$ reads
\begin{align*}
 & i_{\xi}d\Lambda_{\lambda}=(1-n)\frac{1}{\mathscr{L}^{n}}i_{\xi}d\mathscr{L}\bigwedge_{j=1}^{n}\left(\mathscr{L}dx^{j}+\left(\frac{\partial\mathscr{L}}{\partial y_{j}^{\sigma}}-d_{i}\frac{\partial\mathscr{L}}{\partial y_{ij}^{\sigma}}\right)\omega^{\sigma}+\frac{\partial\mathscr{L}}{\partial y_{ij}^{\sigma}}\omega_{i}^{\sigma}\right)\\
 & \quad-(1-n)\frac{1}{\mathscr{L}^{n}}d\mathscr{L}\wedge\sum_{l=1}^{n}(-1)^{l-1}\frac{\partial\mathscr{L}}{\partial y_{jl}^{\sigma}}\xi_{j}^{\sigma}\\
 & \quad\quad\bigwedge_{j\neq l}\left(\mathscr{L}dx^{j}+\left(\frac{\partial\mathscr{L}}{\partial y_{j}^{\sigma}}-d_{i}\frac{\partial\mathscr{L}}{\partial y_{ij}^{\sigma}}\right)\omega^{\sigma}+\frac{\partial\mathscr{L}}{\partial y_{ij}^{\sigma}}\omega_{i}^{\sigma}\right)\\
 & \quad+\frac{1}{\mathscr{L}^{n-1}}\sum_{k=1}^{n}(-1)^{k-1}\left(i_{\xi}d\mathscr{L}dx^{k}+i_{\xi}d\left(\frac{\partial\mathscr{L}}{\partial y_{k}^{\sigma}}-d_{i}\frac{\partial\mathscr{L}}{\partial y_{ik}^{\sigma}}\right)\omega^{\sigma}\right.\\
 & \quad\left.-\left(\frac{\partial\mathscr{L}}{\partial y_{k}^{\sigma}}-d_{i}\frac{\partial\mathscr{L}}{\partial y_{ik}^{\sigma}}\right)\xi_{j}^{\sigma}dx^{j}+i_{\xi}d\frac{\partial\mathscr{L}}{\partial y_{ik}^{\sigma}}\omega_{i}^{\sigma}-\xi_{i}^{\sigma}d\frac{\partial\mathscr{L}}{\partial y_{ik}^{\sigma}}-\frac{\partial\mathscr{L}}{\partial y_{ik}^{\sigma}}\xi_{is}^{\sigma}dx^{s}\right)\\
 & \quad\quad\wedge\bigwedge_{j\neq k}\left(\mathscr{L}dx^{j}+\left(\frac{\partial\mathscr{L}}{\partial y_{j}^{\sigma}}-d_{i}\frac{\partial\mathscr{L}}{\partial y_{ij}^{\sigma}}\right)\omega^{\sigma}+\frac{\partial\mathscr{L}}{\partial y_{ij}^{\sigma}}\omega_{i}^{\sigma}\right)\\
 & \quad+\frac{1}{\mathscr{L}^{n-1}}\sum_{k=1}^{n}(-1)^{k-1}d\left(\mathscr{L}dx^{k}+\left(\frac{\partial\mathscr{L}}{\partial y_{k}^{\sigma}}-d_{i}\frac{\partial\mathscr{L}}{\partial y_{ik}^{\sigma}}\right)\omega^{\sigma}+\frac{\partial\mathscr{L}}{\partial y_{ik}^{\sigma}}\omega_{i}^{\sigma}\right)\\
 & \quad\quad\wedge\sum_{l<k}(-1)^{l-1}\frac{\partial\mathscr{L}}{\partial y_{il}^{\sigma}}\xi_{i}^{\sigma}\bigwedge_{j\neq k,l}\left(\mathscr{L}dx^{j}+\left(\frac{\partial\mathscr{L}}{\partial y_{j}^{\sigma}}-d_{i}\frac{\partial\mathscr{L}}{\partial y_{ij}^{\sigma}}\right)\omega^{\sigma}+\frac{\partial\mathscr{L}}{\partial y_{ij}^{\sigma}}\omega_{i}^{\sigma}\right)
\end{align*}
\begin{align*}
 & \quad+\frac{1}{\mathscr{L}^{n-1}}\sum_{k=1}^{n}(-1)^{k-1}d\left(\mathscr{L}dx^{k}+\left(\frac{\partial\mathscr{L}}{\partial y_{k}^{\sigma}}-d_{i}\frac{\partial\mathscr{L}}{\partial y_{ik}^{\sigma}}\right)\omega^{\sigma}+\frac{\partial\mathscr{L}}{\partial y_{ik}^{\sigma}}\omega_{i}^{\sigma}\right)\\
 & \quad\quad\wedge\sum_{l>k}(-1)^{l}\frac{\partial\mathscr{L}}{\partial y_{il}^{\sigma}}\xi_{i}^{\sigma}\bigwedge_{j\neq k,l}\left(\mathscr{L}dx^{j}+\left(\frac{\partial\mathscr{L}}{\partial y_{j}^{\sigma}}-d_{i}\frac{\partial\mathscr{L}}{\partial y_{ij}^{\sigma}}\right)\omega^{\sigma}+\frac{\partial\mathscr{L}}{\partial y_{ij}^{\sigma}}\omega_{i}^{\sigma}\right).
\end{align*}
Hence the horizontal part of $i_{\xi}d\Lambda_{\lambda}$ satisfies
\begin{align*}
hi_{\xi}d\Lambda_{\lambda} & =(1-n)i_{\xi}d\mathscr{L}\omega_{0}-(1-n)\frac{1}{\mathscr{L}}\sum_{l=1}^{n}\frac{\partial\mathscr{L}}{\partial y_{jl}^{\sigma}}\xi_{j}^{\sigma}d_{l}\mathscr{L}\omega_{0}+ni_{\xi}d\mathscr{L}\omega_{0}\\
 & -\sum_{k=1}^{n}\left(\xi_{i}^{\sigma}d_{k}\frac{\partial\mathscr{L}}{\partial y_{ik}^{\sigma}}+\frac{\partial\mathscr{L}}{\partial y_{ik}^{\sigma}}\xi_{ik}^{\sigma}\right)\omega_{0}-\sum_{k=1}^{n}\left(\frac{\partial\mathscr{L}}{\partial y_{k}^{\sigma}}-d_{i}\frac{\partial\mathscr{L}}{\partial y_{ik}^{\sigma}}\right)\xi_{k}^{\sigma}\omega_{0}\\
 & +\frac{1}{\mathscr{L}}\sum_{k=1}^{n}(-1)^{k-1}d_{s}\mathscr{L}dx^{s}\wedge dx^{k}\wedge\sum_{l<k}(-1)^{l-1}\frac{\partial\mathscr{L}}{\partial y_{il}^{\sigma}}\xi_{i}^{\sigma}\bigwedge_{j\neq k,l}dx^{j}\\
 & +\frac{1}{\mathscr{L}}\sum_{k=1}^{n}(-1)^{k-1}d_{s}\mathscr{L}dx^{s}\wedge dx^{k}\wedge\sum_{l>k}(-1)^{l}\frac{\partial\mathscr{L}}{\partial y_{il}^{\sigma}}\xi_{i}^{\sigma}\bigwedge_{j\neq k,l}dx^{j}\\
 & =-(1-n)\frac{1}{\mathscr{L}}\sum_{l=1}^{n}\frac{\partial\mathscr{L}}{\partial y_{jl}^{\sigma}}\xi_{j}^{\sigma}d_{l}\mathscr{L}\omega_{0}-\frac{1}{\mathscr{L}}\sum_{k=1}^{n}\sum_{l\neq k}\frac{\partial\mathscr{L}}{\partial y_{il}^{\sigma}}\xi_{i}^{\sigma}d_{s}\mathscr{L}dx^{s}\wedge\omega_{l}\\
 & =0,
\end{align*}
where the identity $dx^{k}\wedge\omega_{l}=\delta_{l}^{k}\omega_{0}$
is applied.
\end{proof}
Lepage equivalent $\Lambda_{\lambda}$ \eqref{eq:2nd-Caratheodory}
is said to be the \emph{Carath\'eodory form} associated to $\lambda\in\Omega_{n,X}^{2}W$.

\section{The Carath\'{e}odory form and principal Lepage equivalents in higher-order
theory}

We point out that in the proof of Theorem \ref{thm:Main}, (a), the
chart independence of formula \eqref{eq:2nd-Caratheodory} is based
on principal Lepage equivalent $\Theta_{\lambda}$ \eqref{eq:Poincare-Cartan-2ndOrder}
of a~second-order Lagrangian, which is defined \emph{globally}. Since
for a Lagrangian of order $r\geq3$, principal components of Lepage
equivalents are, in general, \emph{local} expressions (see the Introduction),
we are allowed to apply the definition \eqref{eq:2nd-Caratheodory}
for such class of Lagrangians of order $r$ over a~fibered manifold
which assure invariance of local expressions $\Theta_{\lambda}$ \eqref{eq:PoiCar}.

Consider now a~\emph{third-order} Lagrangian $\lambda\in\Omega_{n,X}^{3}W$.
Then the principal component $\Theta_{\lambda}$ of a~Lepage equivalent
of $\lambda$ reads
\begin{align}
\Theta_{\lambda} & =\mathscr{L}\omega_{0}+\left(\frac{\partial\mathscr{L}}{\partial y_{j}^{\sigma}}-d_{p}\frac{\partial\mathscr{L}}{\partial y_{pj}^{\sigma}}+d_{p}d_{q}\frac{\partial\mathscr{L}}{\partial y_{pqj}^{\sigma}}\right)\omega^{\sigma}\wedge\omega_{j}\label{eq:PrinComp3}\\
 & +\left(\frac{\partial\mathscr{L}}{\partial y_{kj}^{\sigma}}-d_{p}\frac{\partial\mathscr{L}}{\partial y_{kpj}^{\sigma}}\right)\omega_{k}^{\sigma}\wedge\omega_{j}+\frac{\partial\mathscr{L}}{\partial y_{klj}^{\sigma}}\omega_{kl}^{\sigma}\wedge\omega_{j}.\nonumber 
\end{align}
In the following lemma we describe conditions for invariance of \eqref{eq:PrinComp3}.
\begin{lem}
\label{lem:3rdCond}The following two conditions are equivalent:

\emph{(a)} $\Theta_{\lambda}$ satisfies
\[
(\bar{\psi}^{-1}\circ\psi)^{*}\bar{\Theta}_{\lambda}=\Theta_{\lambda}.
\]
\emph{(b)} For arbitrary two overlapping fibered charts on $Y$, $(V,\psi)$,
$\psi=(x^{i},y^{\sigma})$, and $(\bar{V},\bar{\psi})$, $\bar{\psi}=(\bar{x}^{i},\bar{y}^{\sigma})$,
\begin{equation}
d_{k}\left(\left(\frac{\partial\mathscr{L}}{\partial y_{l_{1}l_{2}k}^{\tau}}\frac{\partial x^{s}}{\partial\bar{x}^{p}}-\frac{\partial\mathscr{L}}{\partial y_{l_{1}l_{2}s}^{\tau}}\frac{\partial x^{k}}{\partial\bar{x}^{p}}\right)\frac{\partial^{2}\bar{x}^{p}}{\partial x^{l_{1}}\partial x^{l_{2}}}\frac{\partial y^{\tau}}{\partial\bar{y}^{\sigma}}\right)=0.\label{eq:Obstruction-3rd}
\end{equation}
\end{lem}

\begin{proof}
Equivalence conditions (a) and (b) follows from the chart transformation
\begin{align*}
 & \bar{x}^{k}=\bar{x}^{k}(x^{s}),\quad\bar{y}^{\sigma}=\bar{y}^{\sigma}(x^{s},y^{\nu}),\\
 & \bar{y}_{k}^{\sigma}=\frac{d}{d\bar{x}^{k}}(\bar{y}^{\sigma})=\frac{\partial x^{s}}{\partial\bar{x}^{k}}\frac{\partial\bar{y}^{\sigma}}{\partial x^{s}}+\frac{\partial x^{s}}{\partial\bar{x}^{k}}\frac{\partial\bar{y}^{\sigma}}{\partial y^{\nu}}y_{s}^{\nu},\\
 & \bar{y}_{ij}^{\sigma}=\frac{d}{d\bar{x}^{j}}(\bar{y}_{i}^{\sigma})=\frac{\partial x^{t}}{\partial\bar{x}^{j}}\frac{d}{dx^{t}}\left(\frac{\partial x^{s}}{\partial\bar{x}^{i}}\frac{\partial\bar{y}^{\sigma}}{\partial x^{s}}+\frac{\partial x^{s}}{\partial\bar{x}^{i}}\frac{\partial\bar{y}^{\sigma}}{\partial y^{\nu}}y_{s}^{\nu}\right),\\
 & \bar{y}_{ijk}^{\sigma}=\frac{d}{d\bar{x}^{k}}(\bar{y}_{ij}^{\sigma})=\frac{\partial x^{r}}{\partial\bar{x}^{k}}\frac{d}{dx^{r}}\left(\frac{\partial x^{t}}{\partial\bar{x}^{j}}\frac{d}{dx^{t}}\left(\frac{\partial x^{s}}{\partial\bar{x}^{i}}\frac{\partial\bar{y}^{\sigma}}{\partial x^{s}}+\frac{\partial x^{s}}{\partial\bar{x}^{i}}\frac{\partial\bar{y}^{\sigma}}{\partial y^{\nu}}y_{s}^{\nu}\right)\right),
\end{align*}
applied to $\Theta_{\lambda}$, where Lagrange function $\mathscr{L}$
is transformed by \eqref{eq:LagrangianTransform} and the following
identities are employed,
\begin{align*}
 & \bar{\omega}_{j}=\frac{\partial x^{s}}{\partial\bar{x}^{j}}\det\left(\frac{\partial\bar{x}^{p}}{\partial x^{q}}\right)\omega_{s},\\
 & \bar{\omega}^{\sigma}=\frac{\partial\bar{y}^{\sigma}}{\partial y^{\nu}}\omega^{\nu},\quad\bar{\omega}_{i}^{\sigma}=\frac{\partial\bar{y}_{i}^{\sigma}}{\partial y^{\nu}}\omega^{\nu}+\frac{\partial\bar{y}_{i}^{\sigma}}{\partial y_{p}^{\nu}}\omega_{p}^{\nu},\quad\bar{\omega}_{ij}^{\sigma}=\frac{\partial\bar{y}_{ij}^{\sigma}}{\partial y^{\nu}}\omega^{\nu}+\frac{\partial\bar{y}_{ij}^{\sigma}}{\partial y_{p}^{\nu}}\omega_{p}^{\nu}+\frac{\partial\bar{y}_{ij}^{\sigma}}{\partial y_{pq}^{\nu}}\omega_{pq}^{\nu}.
\end{align*}
\end{proof}
\begin{thm}
\label{thm:Main-3rd}Suppose that a~fibered manifold $\pi:Y\rightarrow X$
and a~non-vanishing third-order Lagrangian $\lambda\in\Omega_{n,X}^{3}W$
satisfy condition \emph{\eqref{eq:Obstruction-3rd}}. Then $\Lambda_{\lambda}$
\emph{\eqref{eq:2nd-Caratheodory}}, where $\Theta_{\lambda}$ is
given by \emph{\eqref{eq:PrinComp3},} defines a~differential $n$-form
on $W^{5}\subset J^{5}Y$, which is a~Lepage equivalent of $\lambda$,
decomposable and $\pi^{5,2}$-horizontal. In a~fibered chart $(V,\psi)$,
$\psi=(x^{i},y^{\sigma})$, on W, $\Lambda_{\lambda}$ has an expression

\begin{align}
\Lambda_{\lambda} & =\frac{1}{\mathscr{L}^{n-1}}\bigwedge_{j=1}^{n}\left(\mathscr{L}dx^{j}+\left(\frac{\partial\mathscr{L}}{\partial y_{j}^{\sigma}}-d_{p}\frac{\partial\mathscr{L}}{\partial y_{pj}^{\sigma}}+d_{p}d_{q}\frac{\partial\mathscr{L}}{\partial y_{pqj}^{\sigma}}\right)\omega^{\sigma}\right.\label{eq:3rdCaratheodoryExpression}\\
 & \qquad\qquad\qquad+\left.\left(\frac{\partial\mathscr{L}}{\partial y_{ij}^{\sigma}}-d_{p}\frac{\partial\mathscr{L}}{\partial y_{ipj}^{\sigma}}\right)\omega_{i}^{\sigma}+\frac{\partial\mathscr{L}}{\partial y_{ikj}^{\sigma}}\omega_{ik}^{\sigma}\right)\nonumber 
\end{align}
\end{thm}

\begin{proof}
This is an immediate consequence of Lemma \ref{lem:3rdCond} and the
procedure given by Lemma \ref{lem:Car-PC} and Theorem \ref{thm:Main}.
\end{proof}
\begin{rem}
Note that according to Lemma \ref{lem:3rdCond}, characterizing obstructions
for the principal component $\Theta_{\lambda}$ \eqref{eq:PrinComp3}
to be a~global form, the Carath\'eodory form \eqref{eq:3rdCaratheodoryExpression}
is well-defined for third-order Lagrangians on fibered manifolds,
which satisfy condition \eqref{eq:Obstruction-3rd}. Trivial cases
when \eqref{eq:Obstruction-3rd} holds identically include namely
(i) Lagrangians \emph{independent} of variables $y_{ijk}^{\sigma}$,
and (ii) fibered manifolds with bases endowed by smooth structure
with \emph{linear} chart transformations. An example of (ii) are fibered
manifolds over two-dimensional open \emph{M\"{o}bius strip} (for
details see \cite{UV}).
\end{rem}

Suppose that a~pair $(\lambda,\pi)$, where $\pi:Y\rightarrow X$
is a~fibered manifold and $\lambda\in\Omega_{n,X}^{r}W$ is a Lagrangian
on $W^{r}\subset J^{r}Y$, induces \emph{invariant} principal component
$\Theta_{\lambda}$ \eqref{eq:PoiCar} of a~Lepage equivalent of
$\lambda$, with respect to fibered chart transformations on $W$.
We call $n$-form $\Lambda_{\lambda}$ on $W^{2r-1}$,
\begin{equation}
\Lambda_{\lambda}=\frac{1}{\mathscr{L}^{n-1}}\bigwedge_{j=1}^{n}(-1)^{n-j}i_{d_{n}}\ldots i_{d_{j+1}}i_{d_{j-1}}\ldots i_{d_{1}}\Theta_{\lambda},\label{eq:Caratheodory-rth}
\end{equation}
where $\Theta_{\lambda}$ is given by \eqref{eq:PoiCar}, the \emph{Carath\'eodory
form} associated to Lagrangian $\lambda\in\Omega_{n,X}^{r}W$.

In a~fibered chart $(V,\psi)$, $\psi=(x^{i},y^{\sigma})$, on W,
$\Lambda_{\lambda}$ has an expression

\begin{align}
\Lambda_{\lambda} & =\frac{1}{\mathscr{L}^{n-1}}\bigwedge_{j=1}^{n}\left(\mathscr{L}dx^{j}+\sum_{s=0}^{r-1}\left(\sum_{k=0}^{r-1-s}(-1)^{k}d_{p_{1}}\ldots d_{p_{k}}\frac{\partial\mathscr{L}}{\partial y_{i_{1}\ldots i_{s}p_{1}\ldots p_{k}j}^{\sigma}}\right)\omega_{i_{1}\ldots i_{s}}^{\sigma}\right).\label{eq:HOCaratheodoryExpression}
\end{align}

\section{Example: The Carath\'eodory equivalent of the Hilbert Lagrangian}

Consider a~fibered manifold $\mathrm{Met}X$ of metric fields over
$n$-dimensional manifold $X$ (see \cite{Volna} for geometry of
$\mathrm{Met}X$). In a~chart $(U,\varphi)$, $\varphi=(x^{i})$,
on $X$, section $g:X\supset U\rightarrow\mathrm{Met}X$ is expressed
by $g=g_{ij}dx^{i}\otimes dx^{j}$, where $g_{ij}$ is symmetric and
regular at every point $x\in U$. An induced fibered chart on second
jet prolongation $J^{2}\mathrm{Met}X$ reads $(V,\psi)$, $\psi=(x^{i},g_{jk},g_{jk,l},g_{jk,lm})$.

The \emph{Hilbert Lagrangian }is an odd-base $n$-form defined on
$J^{2}\mathrm{Met}X$ by 
\begin{equation}
\lambda=\mathscr{R}\omega_{0},\label{eq:Hilbert}
\end{equation}
where $\mathscr{R}=R\sqrt{|\det(g_{ij})|}$, $R=R(g_{ij},g_{ij,k},g_{ij,kl})$
is the \emph{scalar curvature} on $J^{2}\mathrm{Met}X$, and $\mu=\sqrt{|\det(g_{ij})|}\omega_{0}$
is the \emph{Riemann volume element}.

The principal Lepage equivalent of $\lambda$ \eqref{eq:Hilbert}
(cf. formula \eqref{eq:Poincare-Cartan-2ndOrder}), reads
\begin{equation}
\Theta_{\lambda}=\mathscr{R}\omega_{0}+\left(\left(\frac{\partial\mathscr{R}}{\partial y_{j}^{\sigma}}-d_{p}\frac{\partial\mathscr{R}}{\partial y_{pj}^{\sigma}}\right)\omega^{\sigma}+\frac{\partial\mathscr{R}}{\partial y_{ij}^{\sigma}}\omega_{i}^{\sigma}\right)\wedge\omega_{j},\label{eq:FLE-Hilbert}
\end{equation}
and it is a~globally defined $n$-form on $J^{1}\mathrm{Met}X$.
\eqref{eq:FLE-Hilbert} was used for analysis of structure of Einstein
equations as a~system of \emph{first-order} partial differential
equations (see \cite{KruStep}). Another Lepage equivalent for a~second-order
Lagrangian in field theory which could be studied in this context
is given by \eqref{eq:2ndCaratheodoryExpression},
\begin{align*}
 & \Lambda_{\lambda}=\frac{1}{\mathscr{R}^{n-1}}\bigwedge_{k=1}^{n}\left(\mathscr{R}dx^{k}+\left(\frac{\partial\mathscr{\mathscr{R}}}{\partial g_{ij,k}}-d_{l}\frac{\partial\mathscr{\mathscr{R}}}{\partial g_{ij,kl}}\right)\omega_{ij}+\frac{\partial\mathscr{\mathscr{R}}}{\partial g_{ij,kl}}\omega_{ij,l}\right),
\end{align*}
where $\omega_{ij}=dg_{ij}-g_{ij,s}dx^{s}$, $\omega_{ij,l}=dg_{ij,l}-g_{ij,ls}dx^{s}$.
Using a~chart expression of the scalar curvature, we obtain 
\begin{align*}
 & \Lambda_{\lambda}=\frac{1}{\mathscr{R}^{n-1}}\bigwedge_{k=1}^{n}\left(\mathscr{R}dx^{k}+\frac{1}{2}\sqrt{g}\left(g^{qp}g^{si}g^{jk}-2g^{sq}g^{pi}g^{jk}+g^{pi}g^{qj}g^{sk}\right)g_{pq,s}\omega_{ij}\right.\\
 & \qquad\qquad\qquad\qquad+\sqrt{g}\left(g^{il}g^{kj}-g^{kl}g^{ji}\right)\omega_{ij,l}\biggr).
\end{align*}
This is the Carath\'eodory equivalent of the Hilbert Lagrangian.

\section*{Acknowledgements}

This work has been completed thanks to the financial support provided
to the VSB-Technical University of Ostrava by the Czech Ministry of
Education, Youth and Sports from the budget for the conceptual development
of science, research, and innovations for the year 2020, Project No.
IP2300031.

\end{document}